\documentclass[12pt,reqno]{amsart}

\usepackage[margin=1in]{geometry}
\usepackage[dvipsnames]{xcolor}
\usepackage[tt=false]{libertine}

\usepackage{amssymb}

\usepackage[varbb]{newpxmath}
\let\savedbigtimes\bigtimes
\let\bigtimes\relax
\usepackage{mathabx} 
\let\bigtimes\savedbigtimes

\usepackage{pgfplots}
\pgfplotsset{compat=1.16}
\usetikzlibrary{arrows.meta,positioning,calc,decorations.pathreplacing}
\usepgfplotslibrary{fillbetween}

\usepackage{graphicx}
\usepackage{enumerate}
\usepackage{bm}

\usepackage{bbm}

\usepackage{verbatim}
\usepackage{hyperref}
\usepackage[capitalize,nameinlink]{cleveref}
\hypersetup{
	colorlinks=true,
	pdfpagemode=UseNone,
    citecolor=OliveGreen,
    linkcolor=NavyBlue,
    urlcolor=black,
	pdfstartview=FitW
}
\usepackage{appendix}
\crefname{appsec}{Appendix}{Appendices}
\usepackage{tikz}

\theoremstyle{plain}
\newtheorem{theorem}{Theorem}[section]
\newtheorem{proposition}[theorem]{Proposition}
\newtheorem{lemma}[theorem]{Lemma}

\theoremstyle{definition}

\newtheorem*{assumption*}{Assumption}

\theoremstyle{remark}

\crefname{lemma}{Lemma}{Lemmas}
\crefname{theorem}{Theorem}{Theorems}
\crefname{definition}{Definition}{Definitions}
\crefname{fact}{Fact}{Facts}
\crefname{claim}{Claim}{Claims}
\crefname{proposition}{Proposition}{Propositions}

\newcommand{\E}{\mathbb{E}}

\DeclareMathOperator*{\argmax}{arg\,max}

\makeatletter
\newcommand{\vast}{\bBigg@{4}}
\newcommand{\Vast}{\bBigg@{5}}
\makeatother

\renewcommand{\epsilon}{\varepsilon}

\newcommand{\PP}{\mathbb{P}}
\newcommand{\EE}{\mathbb{E}}

\newcommand{\beq}{\begin{equation}}
\newcommand{\eeq}{\end{equation}}

\begin{document}

\title[A Bayesian Proof of Talagrand's Majorizing Measure Theorem]{A Bayesian Proof and Interpretation \\ of Talagrand's Majorizing Measure Theorem}

\author
[Ilias Zadik]{Ilias Zadik$^{\circ}$}

\thanks{\raggedright$^\circ$Department of Statistics and Data Science, Yale University.\\
Email: \texttt{ilias.zadik@yale.edu}}

\date{\today}


\maketitle

\begin{abstract}
In this paper, we give a short Bayesian proof of Talagrand’s celebrated majorizing-measure theorem (MMT). While the upper-bound direction of MMT follows relatively directly from standard arguments, the lower-bound direction is widely regarded as the more difficult part and has received several distinct proofs. Unlike previous approaches, our proof does not rely on existing Gaussian processes lower bounds techniques, nor on combinatorial, geometric, or coding-theoretic constructions. Instead, we derive the lower bound from two area identities for Gaussian additive models. We show that the Gaussian width of a finite set is the integrated mean-squared error of the maximum-likelihood estimator (MLE), while the integrated minimum mean-squared error (MMSE) is larger than the Fernique--Talagrand functional, up to a universal constant. Simply then comparing the MLE with Bayes-optimal estimation, combined with a recent duality minimax argument by Liu, gives a direct proof of the hard direction of MMT.

\end{abstract}

\section{Introduction}
 Talagrand in \cite{Talagrand1987} famously proved the majorizing measure theorem. The importance of this theorem is widely highlighted across the probability theory literature; see, for instance, the discussion surrounding Talagrand's 2024 Abel Prize \cite{guedon2024talagrand}. The theorem is as follows.

\begin{theorem}\label{thm:tal}\cite{Talagrand1987}For any centered and separable Gaussian process $(G_t)_{t \in T}$ when $T$ is endowed with the canonical pseudo-metric \[
    d(s,t)^2=\E(G_s-G_t)^2
\]  then it holds for universal constants $0<c<C$ that \[
    c\mathcal{M}(T,d) \leq \E\sup_{t\in T}G_t \leq C\mathcal{M}(T,d),
\]where $\mathcal{M}(T,d)$ is the Fernique-Talagrand functional
\begin{align}\label{eq:maj}
    \mathcal{M}(T,d)=\inf_{\mu\in\mathcal{P}(T)}\sup_{t\in T}
    \int_0^{\mathrm{diam}(T)}\sqrt{\log\frac1{\mu(B(t,r))}}\,dr
\end{align} 
\end{theorem}


The original proof of this celebrated theorem has often been viewed as opaque, and a substantial effort has gone into finding alternative proofs that offer further insight. As a result, several distinct and very interesting proofs of the majorizing-measure theorem are now available. Since the upper-bound direction follows from relatively direct arguments, these proofs focus mainly on the more challenging lower-bound direction. Talagrand himself gave multiple proofs based on greedy combinatorial constructions together with the Sudakov minoration theorem, leading to the important framework of generic chaining \cite{Talagrand1987,Talagrand1992,Talagrand1996}. Later, van Handel gave a short interpolation proof of the lower bound using a contraction principle \cite{vanHandel2016, vanHandel2018}, while Borst et al.~\cite{Borst2021} developed a related approach based on convex optimization and primal-dual tree certificates. More recently, the theorem was recast in coding-theoretic terms, leading to a proof based on variable-length multiscale codes, Kraft's inequality, and Sudakov-type lower bounds \cite{ChuRaginskyISIT2023}. Most recently, Liu developed a rate-distortion equivalent of the functional \(\mathcal M(T,d)\) and proved the theorem using a lifting method together with Fernique's sharpness of Dudley's entropy integral for stationary processes \cite{Fernique1975,Liu2025}.

In this paper, we present a new proof \emph{of the lower bound} based on a Bayesian statistical argument. Our proof does not rely on Gaussian-process lower-bound tools such as Sudakov minoration, nor does it involve a combinatorial, geometric, or coding-theoretic construction. Instead, it shows that the theorem follows cleanly by comparing the integrated mean-squared error of the maximum-likelihood estimator (MLE) with the integrated minimum mean-squared error (MMSE) in a simple Bayesian Gaussian additive model. The key technical tools we use are standard Bayesian identities, including the Nishimori identity and the I-MMSE formula \cite{GuoShamaiVerdu}, together with an area identity connecting the mean-squared error of the MLE and the Gaussian width of a set (Proposition \ref{prop:maparea}). Figure~\ref{fig:area} illustrates the main skeleton of the proof.

\section{The proof}

\subsection{Getting started}\qquad

\subsubsection{Suffices to consider finite $T$.}
Our proof first notices that by separability of $T$, we may assume $T$ is finite which we assume from now on. This appears to be standard in the literature, but, for completeness, we include here the full reduction argument in Appendix \ref{sec:finite-to-separable}. Moreover, we may assume without loss of generality there is no $t,s \in T$ with $s \neq t$ such that $G_t=G_s$ almost surely.

In particular, since $T$ can be assumed to be finite, we employ the following lemma to realize the Gaussian process in a finite dimensional Euclidean space.

\begin{lemma}\label{lem:finite_euclidean}
Let $T$ be finite and let $(G_t)_{t\in T}$ be a centered Gaussian process. For $N=|T|,$ there exist $N$ \emph{distinct} vectors $(h_t)_{t\in T}$ in $\mathbb{R}^N$ such that, for $Z\sim N(0,I_N)$,
\[
    (G_t)_{t\in T}\stackrel{d}= (\langle Z,h_t\rangle)_{t\in T},
\]
and
\[
    d(s,t)=\|h_s-h_t\|_2.
    \]
\end{lemma}

\begin{proof}
Let $K=(K(s,t))_{s,t\in T}$ be the covariance matrix of the Gaussian process, given by $K(s,t)=\E G_sG_t$ for $s,t \in T$.  Since $K$ is positive semidefinite, there exists a matrix $A \in \mathbb{R}^{N \times N}$ such that $K=AA^\top$. We then set $h_t$ to be the row vector of $A$ indexed by $t$.

Note that by definition for all $s,t \in T$ it holds $\langle h_s,h_t\rangle=K(s,t)$. Therefore, the centered Gaussian process $(\langle Z,h_t\rangle)_{t\in T}$ has the same covariance with the centered $(G_t)_{t\in T}$, and therefore the same law. 

Finally, observe that for all $s,t \in T,$
\[
    \|h_s-h_t\|_2^2=K(s,s)+K(t,t)-2K(s,t)=\E(G_s-G_t)^2=d(s,t)^2.
\]
\end{proof}
Using the above lemma, for the rest of the proof we may assume that the process is represented by vectors in Euclidean space. In fact, without loss of generality, we identify each index $t\in T$ with its corresponding vector, and hence assume from now on that $T\subseteq \mathbb{R}^N$ and
\[
    G_t=\langle Z,t\rangle .
\]
Thus the object of interest is the Gaussian width of the convex hull of $T$, and we denote
\[\mathcal{W}(T)=\E\sup_{t\in T}G_t=\E\sup_{t\in T} \langle Z,t\rangle.\]

\subsubsection{ Rate-distortion bound on $\mathcal{M}(T,d)$ }
Next, we leverage a convenient upper bound (up to constants) on the $\mathcal{M}(T,d)$-functional. This very interesting ``duality" connection, while Liu mentions that might be known before \cite{Liu2025} in the literature, to the best of our knowledge is first proved in \cite{Liu2025}. Specifically, in \cite[Section 6]{Liu2025} it is proven that 
for universal constants $0<c',C'$ that \begin{align}\label{eq:liu}
    c'\mathcal{M}(T,d)-C'\mathrm{diam}(T) \leq \sup_{\pi \in \mathcal{P}(T)} \int_0^{\mathrm{diam}(T)} \sqrt{R_\pi(r)}\,dr,
\end{align} where for $0\le r\le \mathrm{diam}(T)$ and any $\pi \in \mathcal{P}(T)$ we define (a self-coupling version of) the rate-distortion function
\[
    R_\pi(r)
    =\inf\bigl\{ I(V;\widehat V): V \sim \pi, \widehat{V} \sim \pi,
        \E\|V-\widehat V\|^2\le r^2\bigr\},
\]
where the infimum is over all couplings of $(V,\widehat V)$\footnote{As customary, in the definition of $ R_\pi(r)$ and throughout the paper, we denote by $H(V)$ the Shannon entropy of a discrete random variable $V$, and by $I(V;\widehat V)=H(V)-H(V|\widehat V)$ the mutual information between two discrete random variables $V$ and $\widehat V$.}. We remark that the proof follows from elementary (but very elegant) calculus arguments and Sion's minimax duality. For reader's convenience, we include the proof in Appendix \ref{sec:integral}.

Now we turn to the following useful elementary observation, which is also stated in \cite[Section 6]{Liu2025} without proof. We prove it here for completeness. Importantly, this observation allows one to ignore the $\mathrm{diam}(T)$-slack term in \eqref{eq:liu}.

\begin{lemma}\label{lem:max}
If \((T,d)\) is a finite metric space then
\[
\mathcal{W}(T)=\mathbb E \sup_{t\in T}G_t \ge \frac{1}{\sqrt{2\pi}}\,\operatorname{diam}(T).
\]
\end{lemma}

\begin{proof}
Write \(D=\operatorname{diam}(T)\). If \(D=0\), the claim is trivial. Assume
\(D>0\). Since \(T\) is finite, there exist \(a,b\in T\) such that
\(d(a,b)=D\). Then
\[
\mathbb E \sup_{t\in T}G_t
\ge \mathbb E \max\{G_a,G_b\}.
\]
Notice
\[
\mathbb E \max\{G_a,G_b\}
=\frac12 \mathbb E (G_a+G_b)+
\frac12 \mathbb E |G_a-G_b|=\frac12 \mathbb E |G_a-G_b|.
\]
But \(G_a-G_b\) is a centered Gaussian random variable with variance $\mathbb E(G_a-G_b)^2=d(a,b)^2=D^2.$
Therefore, $\mathbb E |G_a-G_b|
=
D\,\mathbb E |g|
=
D\sqrt{\frac{2}{\pi}},$
where \(g\sim N(0,1)\). 
\end{proof}

Hence, combining \eqref{eq:liu} and Lemma \ref{lem:max}, to conclude the desired lower bound of Theorem \ref{thm:tal}, it suffices to prove for some universal constant $c_0>0$ and for any $\pi \in \mathcal{P}(T),$

\begin{align}\label{eq:goal}
    c_0 \int_0^{\mathrm{diam}(T)} \sqrt{R_\pi(r)}\,dr  \leq \mathcal{W}(T)=\E\sup_{t\in T}\langle Z,t\rangle. \end{align}We describe now a Bayesian proof of \eqref{eq:goal} for $c_0=\frac{1}{2}.$

\subsection{The statistical model: MLE-width and MMSE-rate-distortion ``area" theorems}\qquad

\subsubsection{The Bayesian statistical model} To prove \eqref{eq:goal} we fix any $\pi \in \mathcal{P}(T)$ and construct a Bayesian Gaussian additive model that $\pi$ plays the role of the \emph{prior}. Specifically, for any signal-to-noise ratio (SNR) $s\ge0$, we assume that the ``signal" $X$ is chosen from the prior $X \sim \pi$ and a statistician observes
\begin{align}\label{eq:GAM}
    Y_s=sX+Z,
\end{align}where $Z \sim N(0,I_N)$. The goal of the statistician is to design an estimator that recovers $X$ from the ``noisy" $Y_s.$

It will be useful for us to focus on the mean-squared error performance of appropriately chosen estimators. For this reason, we define here for any estimator $\hat{A}: \mathbb{R}^N \rightarrow \mathbb{R}^N$ its mean squared error by \[\mathrm{MSE}_s(\hat A)=\EE \| X-\hat{A}(Y_s)\|^2_2.\]

\subsubsection{ A MLE-width area theorem}
We first focus on the so-called maximum likelihood estimator (MLE) 
\[
    \widehat X_s^{\rm MLE}\in
    \argmax_{u\in T} \log\PP(Y_s|u)  = \argmax_{u\in T}
    \left\{\langle Y_s,u\rangle-\frac{s}{2}\|u\|^2\right\},
\]where ties are broken arbitrarily.

It is well-understood in the statistical literature that the performance of convex relaxations of the MLE relates to the Gaussian width of various convex sets, see e.g., the influential works \cite{chatterjee2014new, Tropp2015}.  A key observation in this work is that \emph{for all Gaussian additive models} (i.e., for any prior and any finite $T$) the Gaussian width $\mathcal{W}(T)=\E\sup_{t\in T}\langle Z,t\rangle$ is in fact \emph{equal to} the integrated mean-squared performance of the MLE across all SNR values \footnote{After completing this paper, the author became aware of the recent independent work of Pathak and Zhivotovskiy \cite{pathak2026gaussian}, which proves a closely related area identity for Gaussian width. In particular, their Theorem~2.2 connects the Gaussian width of a closed convex set with the integrated mean-squared error of the least-squares estimator (LSE) over rescalings of that set. Applied, for each \(x\in T\), to the convex set \(\operatorname{conv}(T)-x\), their result gives an alternative route to the MLE area identity used in our Bayesian proof: one can replace Proposition~\ref{prop:maparea}, which concerns the MLE over the finite class \(T\), by the corresponding LSE area identity over the convex hull of \(T\), and then proceed with comparing the LSE over the convex hull of \(T\) with the Bayes-optimal estimator in the next section.}.

\begin{proposition}[Width-MLE area identity]\label{prop:maparea}
For every prior $\pi$ on $T$,
\[
   \mathcal{W}(T)=\frac12\int_0^\infty
   \mathrm{MSE}_s(\widehat X_s^{\rm MLE})\,ds.
\]
\end{proposition}

\begin{proof}
 By expanding $Y_s$ notice that almost surely \[
    \widehat X_s^{\rm MLE}\in  \argmax_{u\in T}
    \left\{\langle Z,u-X\rangle
        -\frac{s}{2}\|u-X\|^2\right\}.
\]

For this reason, fix any $x\in T$ and $z\in \mathbb{R}^N$ and define the function
\[
    \Phi_{x,z}(s)=
    \max_{u\in T}
    \left\{
        \langle z,u-x\rangle
        -\frac{s}{2}\|u-x\|^2
    \right\}, s \geq 0.
\]In particular, if for some $s \geq 0,$ \begin{align}\label{eq:max}\hat{u}_s \in \argmax_{u\in T}
    \left\{
        \langle z,u-x\rangle
        -\frac{s}{2}\|u-x\|^2
    \right\},\end{align} then it holds $\Phi_{x,z}(s)=\langle z,\hat{u}_s-x\rangle
        -\frac{s}{2}\|\hat{u}_s-x\|^2.$

Now, we explain some analytic properties of $\Phi_{x,z}(s)$. Since we can always choose $u=x$ it holds $\Phi_{x,z}(s)\ge0$ for all $s \geq 0$. Also, because $T$ is finite, clearly $\lim_{s \rightarrow +\infty}\Phi_{x,z}(s)= 0$.  Moreover, the function $\Phi_{x,z}(s), s \geq 0$ is the maximum of finitely many linear functions, therefore it is convex and piecewise linear.  Finally, for every $s \geq 0$ that the function $\Phi_{x,z}(s)$ is differentiable, using Danskin's theorem we have for $\widehat u_s$ from \eqref{eq:max} that the derivative satisfies
\[
    (\Phi_{x,z})'(s)=-\frac12\|\widehat u_s-x\|_2^2
\]and also for every $s \geq 0$ the right derivative satisfies $|(\Phi_{x,z})'_{+}(s)| \leq \max_{u \in T} \frac12\|u-x\|_2^2<\infty$. Combining the above,
\[
    \Phi_{x,z}(0)=\Phi_{x,z}(0)-\lim_{s \rightarrow +\infty}\Phi_{x,z}(s)=\frac12\int_0^\infty \|\widehat u_s-x\|_2^2\,ds.
\]
Now set $x=X$ and $z=Z$ and take expectations on the above equality.  The left hand side becomes
\[
    \E\max_{u\in T}\langle Z,u-X\rangle
    =\E\max_{u\in T}\langle Z,u\rangle-\E\langle Z,X\rangle
    =\E\max_{u\in T}\langle Z,u\rangle=\mathcal{W}(T),
\]
because $Z$ is independent of $X$ and centered. Moreover, under this choice of $x,z$ the maximizer can be taken to satisfy for all $s \geq 0,$ $\widehat u_s=  \widehat X_s^{\rm MLE}$ almost surely. Hence, the right hand side becomes equal to $\frac12\int_0^\infty
    \E \|X-\widehat X_s^{\rm MLE}\|_2^2\,ds=\frac12\int_0^\infty
   \mathrm{MSE}_s(\widehat X_s^{\rm MLE})\,ds$, which completes the proof. 
\end{proof}

\subsubsection{An MMSE-rate distortion area theorem}

\begin{figure}[t!]
    \centering
    \begin{tikzpicture}

\begin{axis}[
    width=12cm,
    height=6.2cm,
    xmin=0, xmax=7.8,
    ymin=0, ymax=1.05,
    axis lines=left,
    xlabel={SNR $s$},
    ylabel={error},
    xtick=\empty,
    ytick=\empty,
    clip=false,
    domain=0:5.3,
    samples=200,
    xlabel style={font=\small},
    ylabel style={font=\small},
]

\addplot[name path=base, draw=none] coordinates {(0,0) (5.3,0)};

\addplot[
    name path=map,
    very thick,
    black
]
{0.92*exp(-0.42*x) + 0.05*exp(-1.8*x)};

\addplot[
    name path=mmse,
    very thick,
    black,
    dashed
]
{0.68*exp(-0.68*x) + 0.08*exp(-2.4*x)};

\addplot[
    gray!18
] fill between[of=map and base];

\addplot[
    gray!38
] fill between[of=mmse and base];

\addplot[
    very thick,
    black
]
{0.92*exp(-0.42*x) + 0.05*exp(-1.8*x)};

\addplot[
    very thick,
    black,
    dashed
]
{0.68*exp(-0.68*x) + 0.08*exp(-2.4*x)};

\node[anchor=west, font=\small] at (axis cs:5.45,0.36)
    {solid: $\mathrm{MSE}_{\mathrm{MLE}}(s)$};

\node[anchor=west, font=\small] at (axis cs:5.45,0.20)
    {dashed: $\mathrm{MMSE}_{\pi}(s)$};

\draw[-{Latex[length=2.5mm]}, thick]
    (axis cs:1.10,0.55) -- (axis cs:4.35,0.80);

\node[anchor=west, align=left, font=\small] at (axis cs:4.45,0.81)
{
$\displaystyle
\int_0^\infty
\mathrm{MSE}_{\mathrm{MLE}}(s)\,ds
= 2 \mathcal{W}(T)$
};

\draw[-{Latex[length=2.5mm]}, thick]
    (axis cs:1.65,0.18) -- (axis cs:4.35,0.54);

\node[anchor=west, align=left, font=\small] at (axis cs:4.45,0.55)
{
$\displaystyle
\int_0^\infty
\mathrm{MMSE}_{\pi}(s)\,ds \geq$ $c\int_0^{\mathrm{diam}(T)} \sqrt{R_{\pi}(r)}\,dr$
};


\end{axis}

\end{tikzpicture}
\caption{A pictorial representation of the Bayesian proof.}
\label{fig:area}
\end{figure}
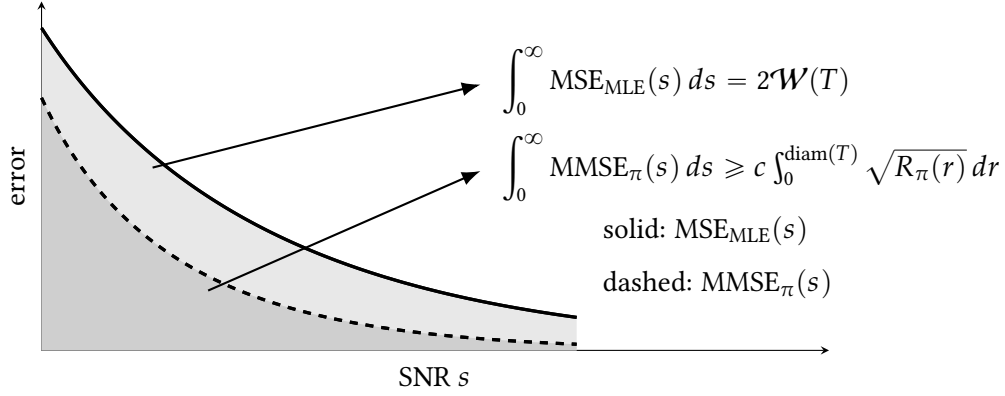

Now, that we know the Gaussian width is equal to the integrated MSE of the MLE, we turn to the performance of the \emph{optimal} Bayesian estimator that minimizes the MSE, which is the posterior mean. To analyse its optimal performance, we use the celebrated I-MMSE formula for our observation Gaussian additive model. 
Specifically, consider the mutual information between the signal and the observations, given by
\[
    I_{\pi}(s)=I(X;Y_s),
\]
and the minimum mean squared error (MMSE), achieved by the posterior mean $\E[X\mid Y_s],$ given by
\[
    \mathrm{MMSE}_{\pi}(s)=\min_{A} \mathrm{MSE}_s(A)=\E\|X-\E[X\mid Y_s]\|^2.
\]
The I--MMSE identity\footnote{Notice that in this work we introduce the I-MMSE formula in a reparametrized form compared to the original version in \cite{GuoShamaiVerdu}, solely because we define our Gaussian additive model with SNR equal to $s$ while often in the literature the SNR of a Gaussian additive model is $\sqrt{s}.$ The reason we make this choice is that this reparametrization of the SNR is more convenient in the analysis of the MLE.} of Guo--Shamai--Verdu \cite{GuoShamaiVerdu} states that for all $s\ge0,$
\begin{equation}\label{eq:immse}
    I'_{\pi}(s)=s \mathrm{MMSE}_{\pi}(s).
\end{equation}

Our first observation is that the MMSE serves as an upper bound to the inverse rate distortion function, defined by
\[
    D_\pi(u)=\inf\bigl\{(\E\|V-\widehat V\|^2)^{1/2}: V \sim \pi, \hat V \sim \pi,
    I(V;\widehat V)\le u\bigr\}
\]where the infimum is again over all couplings $(V,\hat V)$.  Notice $D_\pi$ is non-increasing and $D_\pi(u)=0$ if and only if $u\ge H(V)$.
\begin{lemma}\label{lem:mmseRD}
For every $s\ge0$, it holds
\[
    2\mathrm{MMSE}_{\pi}(s)\ge D_\pi(I_{\pi}(s))^2.
\]
\end{lemma}

\begin{proof}
For $V \sim \pi$, let $Y_s(V)=sV+Z, Z \sim N(0,I_N)$. The posterior mean $\E[V\mid Y_s(V)]$ achieves by definition mean squared error $\mathrm{MMSE}_{\pi}(s)$. Now, by Nishimori's identity (see e.g., \cite[Lemma 2]{niles2023all}), a sample $\hat V$ from the posterior of $V$ given $Y_s(V)$  has mean squared error $2\mathrm{MMSE}_{\pi}(s)$. Moreover, marginally both $V$ and $\hat{V}$ follow $\pi.$ Finally, by data processing,
\[
    I(V;\hat V)\le I(V;Y_s)=I_{\pi}(s).
\]
Therefore $(2\mathrm{MMSE}_{\pi}(s))^{1/2}\ge D_\pi(I_{\pi}(s))$.
\end{proof}

With this lemma at hand we move to the following important step, which relates the integrated MMSE to the integral of the inverse rate distortion function.
\begin{lemma}[MMSE area lower bound]\label{lem:mmsearea}
It holds
\[
    \int_0^\infty \mathrm{MMSE}_{\pi}(s)\,ds
    \ge  \frac1{2}\int_0^{H(V)}\frac{D_\pi(A)}{\sqrt A}\,dA.
\]
\end{lemma}

\begin{proof} Notice that for $0 \leq s < S:=\sup\{ u \geq 0: \mathrm{MMSE}_{\pi}(u)>0\}$ the function $I_{\pi}(s)$ is strictly increasing ranging from $0$ to $H(\pi).$
Hence, by the I-MMSE relation \eqref{eq:immse} and standard change of variables,
\begin{align}\label{eq:step_1}
    \int_0^\infty \mathrm{MMSE}_{\pi}(s)\,ds=   \int_0^S \mathrm{MMSE}_{\pi}(s)\,ds
    =\int_0^S \frac{I'_{\pi}(s)}{s}\,ds=\int_0^{H(\pi)} \frac{1}{I_{\pi}^{-1}(A)}\,dA.
\end{align}
  
By Lemma \ref{lem:mmseRD} and \eqref{eq:immse}, for all $0 \leq s <S,$
\[
    2I'_{\pi}(s)\ge s D_\pi(I_{\pi}(s))^2.
\]
Equivalently, for all $0 \leq u < H(\pi)$,
\[
   (I_{\pi}^{-1})'(u)I_{\pi}^{-1}(u) D_{\pi}(u)^2\le 2.
\]Since $D_\pi$ is nonincreasing, for any $H(V)>A \geq 0$, $D_\pi(u)\ge D_\pi(A)>0$, hence for all $u \leq A$
\[
  ((I_{\pi}^{-1}(u))^2)'\le \frac{4}{D_\pi(A)^2}.
\] By integrating $u$ from $0$ to $A$, the above displayed equation gives for any $H(V)>A> 0$,
\[
  (I_{\pi}^{-1}(A))^2\le \frac{4A}{D_\pi(A)^2}.
\]
or
\[
    \frac1{I_{\pi}^{-1}(A)}\ge \frac{D_\pi(A)}{2\sqrt{A}}.
\]
It follows then from \eqref{eq:step_1}
\[
    \int_0^\infty \mathrm{MMSE}_{\pi}(s)\,ds
    \ge \frac1{2}\int_0^{H(V)}\frac{D_\pi(A)}{\sqrt A}\,dA.
\]
\end{proof}

We now need a final lemma that allows to relate the integral of the inverse rate density function to the integral of the rate density function itself. Satisfyingly, via simple double counting, the exact desired integral appears.
\begin{lemma}\label{lem:layer_cake}
It holds
\[
  \int_0^{H(V)} \frac{D_{\pi}(A)}{\sqrt A}\,dA=2\int_0^{\mathrm{diam}(T)} \sqrt{R_{\pi}(r)}\,dr.
\]
\end{lemma}

\begin{proof}
By standard change of variables and exchanging the order of integration,

\begin{align*}
\int_0^{H(V)} \frac{D_{\pi}(A)}{\sqrt A}\,dA&=2\int_0^{\sqrt{H(V)}} D_{\pi}(u^2)\,du\\
&=2\int_0^{\sqrt{H(V)}}\int_0^{\mathrm{diam}(T)} 1(D_{\pi}(u^2)>r) \,dr\,du\\
&=2\int_0^{\mathrm{diam}(T)}\int_0^{\sqrt{H(V)}}1(R_{\pi}(r)>u^2) \,du\,dr\\
&=2\int_0^{\mathrm{diam}(T)} \sqrt{R_{\pi}(r)}\,dr.
\end{align*}
\end{proof}

\subsection{Putting it all together}

By Proposition \ref{prop:maparea} and the definition of the MMSE as the \emph{minimum} mean squared error among all estimators,
\[
   \mathcal{W}(T)=\E\sup_{t\in T}\langle Z,t\rangle=\frac12\int_0^\infty
   \mathrm{MSE}_s(\widehat X_s^{\rm MLE})\,ds \ge \frac{1}{2}\int_0^\infty \mathrm{MMSE}_{\pi}(s)\,ds.
\]

Applying then Lemma \ref{lem:mmsearea} and Lemma \ref{lem:layer_cake} gives \eqref{eq:goal}.

\section{ Bayesian Intuition and Concluding Thoughts}

In this final section, we highlight a few conceptual consequences of the Bayesian proof in terms of understanding Theorem~\ref{thm:tal}.

As discussed above, the upper-bound direction of Theorem~\ref{thm:tal} is often viewed as the intuitive part of the theorem, see, for example, the generic chaining formulation in \cite{TalagrandBook}.  The lower-bound direction is much less transparent.  One benefit of the present proof is that it gives this direction a clean statistical interpretation.

Let \(T\) be finite, and consider the Gaussian additive model \eqref{eq:GAM}.  Define the integrated Bayes-risk functional
\begin{align}\label{eq:mmse}
    \mathcal{Z}(T,d):=\sup_{\pi \in\mathcal{P}(T)}
    \int_0^\infty \mathrm{MMSE}_{\pi}(s)\,ds .
\end{align}
By Lemma~\ref{lem:mmsearea}, inequality~\eqref{eq:liu}, and an easy MMSE-diameter relation described in Lemma~\ref{lem:mmse-area-diameter}, we obtain for a universal constant \(c>0\) the relation
\[
    c\mathcal{M}(T,d)\le \mathcal{Z}(T,d),
\]
where \(\mathcal{M}(T,d)\) denotes the Fernique--Talagrand functional.  In words, the supremum over $\pi$ of the integrated MMSE functional dominates, up to constants, the classical majorizing-measure functional. The hard direction of Theorem~\ref{thm:tal} then follows from a simple statistical observation: the Bayes estimator is optimal for squared error.  Indeed, for every prior \(\pi\) and every \(s\),
\[
    \mathrm{MMSE}_{\pi}(s)
    \le
    \mathrm{MSE}_s(\widehat X_s^{\rm MLE}).
\]
Moreover, Proposition~\ref{prop:maparea} shows that the area under the MLE error curve is exactly the Gaussian width of \(T\).  In short, the lower bound of the MMT follows from simply comparing the maximum-likelihood estimator with the Bayes-optimal estimator in a Gaussian additive model.

This also gives a canonical interpretation of the optimizing measure.  The classical majorizing measure in \(\mathcal{M}(T,d)\) is an object that is often considered hard to understand probabilistically, see e.g., the discussion in \cite[Section 2]{van2025subgaussian} and how this difficulty has affected the literature of the problem. By contrast, the measure \(\pi\) appearing in the ``dual" \eqref{eq:mmse} is a very canonical statistical object; it is a least favorable prior for the Gaussian additive model, in the sense that it maximizes the integrated Bayes risk; see, for example, \cite[Section~5.3.2]{berger2013statistical} for background on least favorable priors.  Thus, while majorizing measures themselves can be difficult to interpret probabilistically, the dual optimal measure $\pi$ has a clean statistical meaning.


It is perhaps striking that the Bayesian framework fits so naturally into this classical problem. A closely related Bayesian viewpoint was recently used by Mossel, Niles-Weed, Sun, and the author \cite{mossel2025bayesian} to give a new proof of a seemingly quite different result: the fractional Kahn--Kalai conjecture \cite{FKNP} in probabilistic combinatorics, posed by Talagrand \cite{talagrand2010many} as a refinement of earlier conjectures by Kahn and Kalai \cite{kahn2007thresholds}. That conjecture gives a formula for thresholds of monotone properties of random subsets, whereas Theorem~\ref{thm:tal} gives a formula for the supremum of a Gaussian process. The success of a similar Bayesian proof strategy in both settings suggests a broader, though speculative, question: whether the majorizing-measure theorem and the theory of expectation thresholds are manifestations of a common mathematical theory.

Another natural direction for future work is to understand how far the Bayesian proof allows one to generalize the result beyond the Gaussian measure. This question is closely related to several central problems in modern probability. For instance, replacing the Gaussian vector in the definition of Gaussian width by a vector with i.i.d. symmetric Bernoulli entries leads to the Bernoulli conjecture, stated for example in Talagrand’s monograph \cite{TalagrandBook} and proved in the breakthrough work of Bednorz and Latała \cite{bednorz2014boundedness}. More generally, understanding suprema indexed by dependent log-concave random vectors remains a major open direction in the field. A recurring obstruction in extending Gaussian arguments to non-Gaussian processes is the lack of a replacement for Sudakov minoration. This is one reason that the Bayesian approach can provide a promising alternative to approach these questions: the proof above does not invoke Sudakov minoration, but proceeds through an area identity for the maximum-likelihood estimator and the I-MMSE formula\footnote{In fact, we point out to the interested reader that the proof presented in this paper can be easily modified to yield a Bayesian proof of Sudakov minoration by replacing Lemma \ref{lem:layer_cake} with first choosing $\pi$ to be the uniform measure over an arbitrary $\epsilon$-packing and then simply applying Fano's inequality to directly lower bound $D_{\pi}(A)$ for $A$ up to constant of the logarithm of the packing's cardinality. In fact, interestingly, the duality step by Liu \eqref{eq:liu} turns out not to be necessary to prove Sudakov minoration via the Bayesian approach.}. While the I-MMSE formula is classical for Gaussian channels, analogous identities are known for Poisson channels \cite{atar2012mutual} and, in various forms, for exponential-family \cite{raginsky2009mutual}, typically with squared error replaced by a channel-specific loss. It is therefore natural to ask whether the present Bayesian mechanism can be adapted to give new insights into non-Gaussian majorizing-measure-type problems. Interestingly, this seems even more promising as only a few days after the first version of this paper was posted online, Pathak and Zhivotovskiy \cite{pathak2026remark} showed that the present Bayesian/I-MMSE proof strategy can be adapted to yield the $M(T,d)$ lower bound (up to constants) for all centered probability measures whose translates satisfy a quadratic Kullback--Leibler stability condition.

\section*{Acknowledgments}
The author is thankful to Michel Talagrand, Ramon van Handel, Jingbo Liu, Reese Pathak, Nikita Zhivotovskiy, Alkis Kalavasis, Jonathan Niles-Weed and Manolis Zampetakis for helpful feedback and comments.

\newpage

\bibliographystyle{alpha}
\bibliography{biblio}

\newpage
\appendix

\section{From finite to separable Gaussian processes}
\label{sec:finite-to-separable}

In the Bayesian proof in the main body we assumed $T$ is finite. While it appears folklore in the literature that the finite-$T$ statement of the hard direction of Talagrand's Theorem \ref{thm:tal} extends to any separable Gaussian process, we include here, for completeness, a full compactness proof establishing the reduction.

Let \((G_t)_{t\in T}\) be any centered separable Gaussian process with canonical metric
\[
d(s,t)=\bigl(\mathbb E(G_s-G_t)^2\bigr)^{1/2}.
\] Now assume for any finite \(F\), we have

\[
\mathbb E\max_{t\in F}G_t
\ge
c_0 \mathfrak R (F)
\]
where \(c_0>0\) is a universal constant and \[
\mathfrak R (F)
:=
\sup_{\pi\in\mathcal P(F)}
\int_0^{\operatorname{diam}(F)}
\sqrt{R_{\pi,F}(r)}\,dr.
\] By monotone convergence we directly get
\[\mathbb E\max_{t\in T}G_t
\ge 
c_0 \mathfrak R_{\rm fin}(T)\]for
\[
\mathfrak R_{\rm fin}(T)
:=
\sup_{\substack{F\subseteq T\\ F\ \mathrm{finite}}}
\;
\sup_{\pi\in\mathcal P(F)}
\int_0^{\operatorname{diam}(F)}
\sqrt{R_{\pi,F}(r)}\,dr.
\]
To continue, we turn to the partition version of \(\gamma_2\) Talagrand's functional, given by
\[
\gamma_2^{\rm part}(T,d)
=
\inf_{\{\mathcal A_n\}}
\sup_{t\in T}
\sum_{n\ge0}
2^{n/2}\operatorname{diam}(\mathcal A_n(t)),
\]where for each $n$ \(\mathcal A_n\) ranges over partitions of \(T\) with number of cells satisfying
\[
|\mathcal A_n|\le N_n=2^{2^n},
\]
and \(\mathcal A_n(t)\) denotes the unique cell of \(\mathcal A_n\) containing \(t\).
It is known in the literature that $\mathfrak R (T)$ and $ \gamma_2^{\rm part}(T,d)$ (as well $\mathcal{M}(T,d)$ from \eqref{eq:maj}) are equal up to universal constants for any metric space $(T,d)$ \cite{Talagrand2005GenericChaining}.  Hence it suffices to prove the following lemma, which follows from an elementary compactness argument. We include the full proof here below.

\begin{lemma}
\label{lem:finite-subset-gamma2}
For every metric space \((T,d)\),
\[
\gamma_2^{\rm part}(T,d)
=
\sup_{\substack{F\subseteq T\\ F\ \mathrm{finite}}}
\gamma_2^{\rm part}(F,d).
\]
\end{lemma}

\begin{proof}
The inequality
\[
\sup_{F\subseteq T,\ |F|<\infty}\gamma_2^{\rm part}(F,d)
\le
\gamma_2^{\rm part}(T,d)
\]
is immediate. Indeed, if \((\mathcal A_n)\) is an admissible sequence of partitions of \(T\), then its restriction to \(F\),
\[
\mathcal A_n|_F
:=
\{A\cap F:\ A\in\mathcal A_n,\ A\cap F\neq \varnothing\},
\]
is an admissible sequence of partitions of \(F\), and for every \(t\in F\),
\[
\operatorname{diam}\bigl((\mathcal A_n|_F)(t)\bigr)
\le
\operatorname{diam}\bigl(\mathcal A_n(t)\bigr).
\]
Taking the infimum over admissible partition sequences on \(T\) gives the claim.

We now prove the reverse inequality. Let
\[
L:=
\sup_{\substack{F\subseteq T\\ F\ \mathrm{finite}}}
\gamma_2^{\rm part}(F,d).
\]
If \(L=+\infty\), there is nothing to prove so we assume \(L<\infty\). Now, fix \(\varepsilon>0\). We shall construct an admissible partition sequence \((\mathcal A_n)\) of \(T\) such that
\[
\sup_{t\in T}
\sum_{n\ge0}
2^{n/2}\,\operatorname{diam}\bigl(\mathcal A_n(t)\bigr)
\le L+\varepsilon.
\]
This will imply
\[
\gamma_2^{\rm part}(T,d)\le L+\varepsilon,
\]
and then the result follows by letting \(\varepsilon\) go to zero.

For each \(n\), let \([N_n]:=\{1,\dots,N_n\}\). Consider the compact product space
\[
\Omega
:=
\prod_{n\ge0}[N_n]^T
\]
with the product topology, where each \([N_n]\) has the discrete topology. A point
\[
\omega=(\ell_n)_{n\ge0}\in\Omega
\]
assigns to each \(n\) a label map
\[
\ell_n:T\to [N_n].
\]
The inverse images of \(\ell_n\) define a partition of \(T\) with at most \(N_n\) cells.

For a finite set \(E\subseteq T\), a point \(t\in E\), an integer \(M\ge0\), and a labeling \(\omega=(\ell_n)\), define
\[
\Delta_n^E(t;\omega)
:=
\operatorname{diam}\{x\in E:\ell_n(x)=\ell_n(t)\}.
\]
Equivalently,
\[
\Delta_n^E(t;\omega)
=
\max\{d(x,y):x,y\in E,\ \ell_n(x)=\ell_n(y)=\ell_n(t)\}.
\]
Since \(E\) is finite, this maximum is over a finite nonempty set.

Now define the closed subset
\[
C(E,t,M)
:=
\left\{
\omega\in\Omega:
\sum_{n=0}^M 2^{n/2}\,\Delta_n^E(t;\omega)
\le L+\varepsilon
\right\}.
\]
The set \(C(E,t,M)\) is closed because it depends only on finitely many labels \(\ell_0,\dots,\ell_M\) restricted to the finite set \(E\).

We claim that the family of closed sets
\[
\{C(E,t,M): E\subseteq T\text{ finite},\ t\in E,\ M\ge0\}
\]
has the finite intersection property.

Indeed, take finitely many constraints
\[
C(E_1,t_1,M_1),\dots,C(E_k,t_k,M_k).
\]
Let
\[
E_\ast:=E_1\cup\cdots\cup E_k.
\]
By the definition of \(L\), there is an admissible partition sequence
\[
(\mathcal B_n)_{n\ge0}
\]
of the finite metric space \(E_\ast\) such that
\[
\sup_{u\in E_\ast}
\sum_{n\ge0}
2^{n/2}\,\operatorname{diam}\bigl(\mathcal B_n(u)\bigr)
\le L+\varepsilon.
\]
Label the cells of \(\mathcal B_n\) by elements of \([N_n]\). This gives maps
\[
\ell_n:E_\ast\to[N_n].
\]
Extend each \(\ell_n\) arbitrarily to all of \(T\), for instance by assigning all points of \(T\setminus E_\ast\) to label \(1\). 

For every \(j=1,\dots,k\), every \(t_j\in E_j\), and every \(n\le M_j\),
\[
\Delta_n^{E_j}(t_j;\omega)
\le
\operatorname{diam}\bigl(\mathcal B_n(t_j)\bigr),
\]
because \(E_j\subseteq E_\ast\). Hence
\[
\sum_{n=0}^{M_j}
2^{n/2}\,\Delta_n^{E_j}(t_j;\omega)
\le
\sum_{n\ge0}
2^{n/2}\,\operatorname{diam}\bigl(\mathcal B_n(t_j)\bigr)
\le L+\varepsilon.
\]
Thus this labeling belongs to all of the finitely many closed sets. The finite intersection property is proved.

Since \(\Omega\) is compact, we conclude the intersection of all the sets \(C(E,t,M)\) is nonempty. Choose
\[
\omega=(\ell_n)_{n\ge0}
\]
in this intersection. Let \(\mathcal A_n\) be the partition of \(T\) into the fibers of \(\ell_n\). Then
\[
|\mathcal A_n|\le N_n.
\]

Now, fix \(t\in T\) and \(M\ge0\). We claim
\[
\sum_{n=0}^M
2^{n/2}\,\operatorname{diam}\bigl(\mathcal A_n(t)\bigr)
\le L+\varepsilon.
\]
For each \(0\le n\le M\), choose points \(x_n,y_n\in \mathcal A_n(t)\) such that
\[
d(x_n,y_n)
\ge
\operatorname{diam}\bigl(\mathcal A_n(t)\bigr)-\delta_n,
\]
where \(\delta_n>0\) will be chosen later. If the diameter is not attained, choose \(x_n,y_n\) approximating the supremum; if the diameter is infinite, the argument below gives an immediate contradiction with $L<\infty$ by choosing pairs with arbitrarily large distance.

Let
\[
E:=\{t\}\cup\{x_n,y_n:0\le n\le M\}.
\]
Since \(\omega\in C(E,t,M)\), we have
\[
\sum_{n=0}^M
2^{n/2}\,\Delta_n^E(t;\omega)
\le L+\varepsilon.
\]
But \(x_n,y_n\in E\) and
\[
\ell_n(x_n)=\ell_n(y_n)=\ell_n(t),
\]
so
\[
\Delta_n^E(t;\omega)\ge d(x_n,y_n)
\ge
\operatorname{diam}\bigl(\mathcal A_n(t)\bigr)-\delta_n.
\]
Therefore
\[
\sum_{n=0}^M
2^{n/2}\,\operatorname{diam}\bigl(\mathcal A_n(t)\bigr)
\le
L+\varepsilon+\sum_{n=0}^M a_n\delta_n.
\]
Letting all \(\delta_n\downarrow0\), we obtain
\[
\sum_{n=0}^M
2^{n/2}\,\operatorname{diam}\bigl(\mathcal A_n(t)\bigr)
\le L+\varepsilon.
\]
Since this holds for every \(M\), monotone convergence of the partial sums yields
\[
\sum_{n\ge0}
2^{n/2}\,\operatorname{diam}\bigl(\mathcal A_n(t)\bigr)
\le L+\varepsilon.
\]
Finally take the supremum over \(t\in T\). Thus
\[
\sup_{t\in T}
\sum_{n\ge0}
2^{n/2}\,\operatorname{diam}\bigl(\mathcal A_n(t)\bigr)
\le L+\varepsilon.
\]
Hence
\[
\gamma_2^{\rm part}(T,d)\le L+\varepsilon.
\]
Letting \(\varepsilon\downarrow0\) gives
\[
\gamma_2^{\rm part}(T,d)\le L.
\]
Together with the first inequality, this proves
\[
\gamma_2^{\rm part}(T,d)
=
\sup_{\substack{F\subseteq T\\ F\ \mathrm{finite}}}
\gamma_2^{\rm part}(F,d).
\]
\end{proof}

\section{The rate-distortion integral upper bounds the \(\mathcal{M}(T,d)\) functional}
\label{sec:integral}

In this section, we include for completeness the following result and a (very) slightly
modified proof from \cite[Section 6]{Liu2025}. The proof is based on elementary
calculus and Sion's minimax theorem.

\begin{theorem}[\cite{Liu2025}]
\label{thm:LiuBridge}
There exist universal constants \(c,C>0\) such that, for every finite metric
space \((T,d)\),
\[
  \sup_{\pi \in \mathcal P(T)}
  \int_0^{\operatorname{diam}(T)} \sqrt{R_\pi(r)}\,dr
  \ge
  c \mathcal{M}(T,d)-C\operatorname{diam}(T).
\]
\end{theorem}

\begin{proof}
Write
\[
    \Delta:=\operatorname{diam}(T).
\]
We employ the following elementary calculus lemma from \cite[Lemma 7]{Liu2025}.
If \(y:[0,\Delta]\to[0,\infty]\) is non-increasing, right-continuous, and
\(y(\Delta)=0\), then there exist universal constants \(0<c_0<C_0\)
such that
\begin{equation}
\label{eq:penalized-correct}
    c_0\int_0^\Delta y(r)\,dr
    \le
    \int_0^\infty
      \inf_{0\le r\le \Delta}
      \left\{\alpha^{-2}r^2+y(r)^2\right\}
    \,d\alpha
    \le
    C_0\int_0^\Delta y(r)\,dr.
\end{equation}

For a fixed prior \(\pi\in\mathcal P(T)\), apply
\eqref{eq:penalized-correct} to
\[
    y(r)=\sqrt{R_\pi(r)},\qquad 0\le r\le \Delta.
\]
Since \(R_\pi(\Delta)=0\), this gives
\begin{equation}
\label{eq:RD-to-penalty}
    \int_0^\Delta \sqrt{R_\pi(r)}\,dr
    \ge
    C_0^{-1}
    \int_0^\infty
    \inf_{0\le r\le \Delta}
    \left\{\alpha^{-2}r^2+R_\pi(r)\right\}
    \,d\alpha.
\end{equation}
By the definition of \(R_\pi\), for every \(\alpha>0\),
\begin{equation}
\label{eq:penalty-coupling}
    \inf_{0\le r\le \Delta}
    \left\{\alpha^{-2}r^2+R_\pi(r)\right\}
    =
    \inf_{\substack{P_{X,X'}:\\ X\sim\pi,\ X'\sim\pi}}
    \left\{
        \alpha^{-2}\mathbb E d(X,X')^2+I(X;X')
    \right\}.
\end{equation}
Indeed, if a coupling has \(\mathbb E d(X,X')^2\le r^2\), then the right-hand
side is bounded above by \(\alpha^{-2}r^2+R_\pi(r)\); conversely, for any
coupling one may choose \(r=(\mathbb E d(X,X')^2)^{1/2}\).

For \(\alpha>0\), \(\mu\in\mathcal P(T)\), and \(x\in T\), define
\[
    \Psi_{\alpha,\mu}(x)
    :=
    \inf_{\nu\in\mathcal P(T)}
    \left\{
        \alpha^{-2}\mathbb E_{Y\sim\nu} d(x,Y)^2
        +
        D_{\mathrm{KL}}(\nu\|\mu)
    \right\}.
\]
Now, by the Gibbs variational formula \cite[Lemma 2.1.]{donsker1975asymptotic}, it is easy to check that for fixed \(\alpha\) and \(x\), the map
\(\mu\mapsto \Psi_{\alpha,\mu}(x)\) is convex. Let also \(K_x\) denote the conditional law of \(X'\) given \(X=x\). Then
\[
    I(X;X')
    =
    \mathbb E_{X\sim\pi}D_{\mathrm{KL}}(K_X\|\pi),
\]where $D_{\mathrm{KL}}$ is the Kullback-Leibler (KL) divergence.
Now, dropping the marginal constraint \(X'\sim\pi\) in
\eqref{eq:penalty-coupling}, we obtain
\begin{equation}
\label{eq:drop-marginal}
\begin{aligned}
&\inf_{\substack{P_{X,X'}:\\ X\sim\pi,\ X'\sim\pi}}
    \left\{
        \alpha^{-2}\mathbb E d(X,X')^2+I(X;X')
    \right\} \\
&\qquad\ge
    \mathbb E_{X\sim\pi}
    \inf_{\nu\in\mathcal P(T)}
    \left\{
        \alpha^{-2}\mathbb E_{Y\sim\nu}d(X,Y)^2
        +
        D_{\mathrm{KL}}(\nu\|\pi)
    \right\} \\
&\qquad=
    \mathbb E_{X\sim\pi}\Psi_{\alpha,\pi}(X).
\end{aligned}
\end{equation}

Therefore
\[
\begin{aligned}
&\sup_{\pi\in\mathcal P(T)}
    \int_0^\infty
    \inf_{0\le r\le \Delta}
    \left\{
        \alpha^{-2}r^2+R_\pi(r)
    \right\}
    d\alpha \\
    &\qquad \geq
    \sup_{\pi\in\mathcal P(T)}
    \sum_{x\in T}\pi(x)
    \int_0^\infty
    \Psi_{\alpha,\pi}(x)\,d\alpha\\
&\qquad\ge
    \sup_{\pi\in\mathcal P(T)}
    \inf_{\mu\in\mathcal P(T)}
    \sum_{x\in T}\pi(x)
    \int_0^\infty
    \Psi_{\alpha,\mu}(x)\,d\alpha.
\end{aligned}
\]
By Sion's minimax theorem, using the convexity in \(\mu\) and linearity in \(\pi\),
\begin{equation}
\label{eq:sion-correct}
\begin{aligned}
&\sup_{\pi\in\mathcal P(T)}
    \inf_{\mu\in\mathcal P(T)}
    \sum_{x\in T}\pi(x)
    \int_0^\infty
    \Psi_{\alpha,\mu}(x)\,d\alpha \\
&\qquad=
    \inf_{\mu\in\mathcal P(T)}
    \sup_{\pi\in\mathcal P(T)}
    \sum_{x\in T}\pi(x)
    \int_0^\infty
    \Psi_{\alpha,\mu}(x)\,d\alpha \\
&\qquad=
    \inf_{\mu\in\mathcal P(T)}
    \sup_{x\in T}
    \int_0^\infty
    \Psi_{\alpha,\mu}(x)\,d\alpha.
\end{aligned}
\end{equation}

It remains to lower bound the last display in terms of \(\mathcal{M}(T,d)\). Fix
\(x\in T\), \(\mu\in\mathcal P(T)\), and \(\alpha>0\). For \(0\le r\le\Delta\),
write
\[
    L_x(r):=\log\frac{1}{\mu(B(x,r))}.
\]
We claim that there exist universal constants \(L_0,c_1>0\) such
that
\begin{equation}
\label{eq:psi-ball-lower}
    \Psi_{\alpha,\mu}(x)
    \ge
    c_1
    \inf_{0\le r\le\Delta}
    \left\{
        \alpha^{-2}r^2+\bigl(L_x(r)-L_0\bigr)_+
    \right\}.
\end{equation}

To prove this, fix \(\nu\in\mathcal P(T)\) and set
\[
    m_2:=\mathbb E_{Y\sim\nu}d(x,Y)^2.
\]
Choose a fixed number \(a>1\), say \(a=2\), and put
\[
    \rho:=\min\{a\sqrt{m_2},\Delta\}.
\]
Then \(m_2\ge \rho^2/a^2\). Markov's inequality gives
\[
    \nu(B(x,\rho))\ge 1-a^{-2}.
\] Now it is easy to check the elementary bound for the binary relative entropy $d_{\mathrm{bin}}(p\|q),$ that for some constants $c_a,L_a>0$ if  $p\ge 1-a^{-2},$
\[
    d_{\mathrm{bin}}(p\|q)
    \ge
    c_a\left(\log\frac1q-L_a\right)_+.
\]Hence, by data processing,
\[
    D_{\mathrm{KL}}(\nu\|\mu)
    \ge
    c_a\bigl(L_x(\rho)-L_0\bigr)_+.
\]
Consequently,
\[
\begin{aligned}
    \alpha^{-2}m_2+D_{\mathrm{KL}}(\nu\|\mu)
    &\ge
    c_1
    \left\{
        \alpha^{-2}\rho^2+\bigl(L_x(\rho)-L_0\bigr)_+
    \right\} \\
    &\ge
    c_1
    \inf_{0\le r\le\Delta}
    \left\{
        \alpha^{-2}r^2+\bigl(L_x(r)-L_0\bigr)_+
    \right\}.
\end{aligned}
\]
Taking the infimum over \(\nu\) proves \eqref{eq:psi-ball-lower}.

Now apply the calculus lemma \eqref{eq:penalized-correct} to the non-increasing
function
\[
    y_x(r):=\sqrt{\bigl(L_x(r)-L_0\bigr)_+},
    \qquad 0\le r\le\Delta.
\]
Since \(L_x(\Delta)=0\), we have \(y_x(\Delta)=0\). Combining
\eqref{eq:psi-ball-lower} with the lower bound in
\eqref{eq:penalized-correct} gives
\[
\begin{aligned}
    \int_0^\infty \Psi_{\alpha,\mu}(x)\,d\alpha
    &\ge
    c_2
    \int_0^\Delta
    \sqrt{\bigl(L_x(r)-L_0\bigr)_+}\,dr \\
    &\ge
    c_2
    \int_0^\Delta
    \sqrt{L_x(r)}\,dr
    -
    C_2\Delta,
\end{aligned}
\]
where we used
\[
    \sqrt{(u-L_0)_+}\ge \sqrt{u}-\sqrt{L_0},
    \qquad u\ge0.
\]
Therefore
\[
\begin{aligned}
    \inf_{\mu\in\mathcal P(T)}
    \sup_{x\in T}
    \int_0^\infty \Psi_{\alpha,\mu}(x)\,d\alpha
    &\ge
    c_2
    \inf_{\mu\in\mathcal P(T)}
    \sup_{x\in T}
    \int_0^\Delta
    \sqrt{\log\frac1{\mu(B(x,r))}}\,dr
    -
    C_2\Delta \\
    &=
    c_2 \mathcal{M}(T,d)-C_2\Delta.
\end{aligned}
\]
Combining this estimate with
\eqref{eq:RD-to-penalty}, \eqref{eq:penalty-coupling},
\eqref{eq:drop-marginal}, and \eqref{eq:sion-correct}, we obtain
\[
    \sup_{\pi\in\mathcal P(T)}
    \int_0^\Delta \sqrt{R_\pi(r)}\,dr
    \ge
    c \mathcal{M}(T,d)-C\Delta
\]
for universal constants \(c,C>0\). This completes the proof.
\end{proof}

\section{Diameter lower bound on the MMSE area }
In this section, we include an auxiliary lemma relating for any Gaussian additive model the MMSE area and the diameter of the parameter space.
\begin{lemma}
\label{lem:mmse-area-diameter}
Let \(T\) be a finite subset of a Euclidean space $\mathbb{R}^N$. For a prior \(\pi\) on \(T\), let
\[
Y_s=sX+Z,\qquad X\sim \pi,\qquad Z\sim N(0,I_N).
\]

Then for some universal constant $c>0,$
\[
\sup_{\pi}
\int_0^\infty \operatorname{MMSE}_\pi(s)\,ds
\ge
c\operatorname{diam}(T).
\]

\end{lemma}

\begin{proof}
Fix two points \(x_0,x_1\in T\), and put
\[
\delta:=\|x_1-x_0\|_2.
\]
We will show that for $\pi$ the uniform prior on \(\{x_0,x_1\}\) we have for some universal constant $c>0,$
\[
\int_0^\infty \operatorname{MMSE}_\pi(s)\,ds
\ge
c\delta.
\]

By translation and rotation, we can assume
\[
x_0=-\frac{\delta}{2}e_1,
\qquad
x_1=\frac{\delta}{2}e_1,
\]
where \(e_1\) is the first coordinate vector. Hence, we have
$X=\frac{\delta}{2}B e_1,$ for
$B\sim \mathrm{Unif}(\{-1,+1\}).$  Thus the problem reduces to the one-dimensional Gaussian channel
\[
Y=\alpha B+N,
\qquad
N\sim N(0,1),
\qquad
\alpha=\frac{s\delta}{2}.
\]
Consider the one-dimensional
\[
m_B(\alpha)
:=
\mathbb E\bigl[(B-\mathbb E[B\mid Y])^2\bigr].
\]and then we have
\[
\operatorname{MMSE}_\pi(s)
=
\frac{\delta^2}{4}\,m_B\!\left(\frac{s\delta}{2}\right).
\]

But direct calculations give that $\mathbb E[B\mid Y]=\tanh(\alpha Y),$
and therefore
\[
m_B(\alpha)
=
\mathbb E\bigl[1-\tanh^2(\alpha Y)\bigr]
=
\mathbb E\bigl[\operatorname{sech}^2(\alpha Y)\bigr] \geq \operatorname{sech}^2(1) \PP(\alpha|Y| \leq 1).
\]
But for \(0\le\alpha\le1\), 
\[
\mathbb P(|Y|\le1)
=
\mathbb P(|N+\alpha|\le1)
\ge
\mathbb P(-2\le N\le 0),
\]so $m_B(\alpha) \geq \operatorname{sech}^2(1) \mathbb P(-2\le N\le 0).$

Therefore for some universal constant $c>0$, for all $0 \leq \alpha \leq 1$,
\[
\begin{aligned}
\int_0^\infty \operatorname{MMSE}_\pi(s)\,ds
&\ge
\int_0^{2/\delta}
\frac{\delta^2}{4}
m_B\!\left(\frac{s\delta}{2}\right)\,ds \\
&\ge
c\int_0^{2/\delta}
\frac{\delta^2}{4}\,ds \\
&=
c\frac{\delta}{2}.
\end{aligned}
\]

\end{proof}
\end{document}